\documentclass[11pt]{amsart}

\usepackage{macros}
\usepackage{mathtools}

\def\co{\colon\thinspace}

\usepackage{macros}
\usepackage{mathtools}
\usepackage{tikz-cd}

\def\co{\colon\thinspace}
\DeclareMathOperator{\id}{id}

\def\op{\mathrm}

\def\calgd{\mathsf{C}^*}
\def\fn{\mathfrak{n}}

\def\xto{\xrightarrow}

\title{Parametrized $\L8$ Spaces}

\author{Ryan Grady}
\address{Perimeter Institute for Theoretical Physics\\Waterloo, Ontario\\Canada N2L 2Y5}
\email{rgrady@perimeterinstitute.ca}
\thanks{The author was partially supported by the National Science Foundation under Award DMS-1309118.}


\begin{document}

\maketitle


\tableofcontents

\section{Introduction and Overview}

In this note, we generalize the notion of $\L8$ space by allowing sheaves of $\L8$ algebras over any nilpotent dg locally ringed space.  

Since at least the work of Lada and Stasheff \cite{LadaJim} or Lada and Markl \cite{LadaMarkl}, $\L8$ algebras have played an important role in both physics and deformation theory.  Motivated by using derived geometry to build $\sigma$-models, Costello defined $\L8$ spaces in \cite{CosWG2}. In joint work with Gwilliam \cite{GG1, GGLoop, GGAlgd}, we used similar techniques to clarify the relationship between $\L8$ spaces and smooth derived stacks, as well as, to provide several further examples, e.g., the $\L8$ space associated to a Lie algebroid.

In the existing literature, an $\L8$ space is a sheaf of curved $\L8$ algebras over the dg manifold $(M, \Omega^\ast_M)$ corresponding to a smooth manifold $M$; in the curved setting, we must also remember the nilpotent ideal $\Omega^{\ge 1}_M$. One should think of an $\L8$ space as a flat family of formal moduli problems glued together over the smooth manifold $M$. Via their associated Maurer-Cartan functors (see \cite{GGLoop}), $\L8$ spaces provide examples of derived stacks.  One can think of them as smooth (and somewhat watered-down) analogues of the {\it formal derived stacks} of Calaque, Pantev, To\"{e}n, Vaqui\'{e}, and Vezzosi \cite{CPTVV}; in particular, compare section 2 of the recent survey of Pantev and Vezzosi \cite{PanGab}. (See also Lurie's work, especially \cite{LurieICM, LurieDAGX}.)

Recently, Tu \cite{Tu1, Tu2} has given another notion of $\L8$ space in his approach to derived geoemetry and the construction of fundamental cycles.  This definition is very similar to that of Costello, but critically does not require that the curvature term vanishes modulo a nilpotent ideal (see Remark 3.3.4 of \cite{Tu1} for the precise comparison).  One of our aims is to quantize various mapping spaces using the BV formalism (a la Costello \cite{Cos1}), for which we need the curvature to be nilpotent.

In the present work, we enlarge the definition of $\L8$ space to include other types of families of formal moduli problems.  By changing the underlying parametrizing dg manifold, we obtain smooth families, holomorphic families, or more generally families parametrized along the leaves of a foliation.  The details are provided in Section \ref{sect:param}. We discuss the associated derived stack in Section \ref{sect:stack}.

In section \ref{sect:example}, we lay out a slew of examples. In particular, we show that Yu's construction of the Dolbeault dga of a formal neighborhood \cite{Yu15} defines a holomorphic $\L8$ space.  In the final section (Section \ref{sect:algd}), we use the data of a Lie algebroid $L \to T_M$ to construct an $\L8$ space over the cochains of $L$.  We show that the characteristic classes of this $\L8$ space recover the primary invariants of the Lie algebroid $L$.

\subsection{Notations and conventions}

We work throughout in characteristic zero.
We work cohomologically, so the differential in any complex increases degree by one.

For $A$ a cochain complex, $A^\sharp$ denotes the underlying graded vector space. If $A$ is a cochain complex whose degree $k$ space is $A^k$, then $A[1]$ is the cochain complex where $A[1]^k = A^{k+1}$. We use $A^\vee$ to denote the graded dual.

For $X$ a smooth manifold, we use $T_X$ to denote its tangent bundle as a vector bundle and $TX$ to denote the total space of that vector bundle. We use $T^\vee_X$ to denote the cotangent bundle.

If $f \co X \to Y$ is a map of smooth manifolds and $V$ a vector bundle, then we use $f^{-1} V$ to denote the pullback vector bundle. Similarly, for $\cF$ a sheaf on $Y$, we use $f^{-1}\cF$ to denote the pullback sheaf, simply as a sheaf of sets or vector spaces. We reserve the notation $f^\ast \cV$ for the case where $\cV$ is a sheaf of dg $\Omega^\ast_Y$-modules, and $f^\ast \cV$ denotes the sheaf of dg $\Omega^\ast_X$-modules obtained from $f^{-1} \cV$ by extending scalars.  

For $\rho \co L \to T_X$ a Lie algebroid, we let $J_L$ denote the sheaf of $L$-jets as defined in Appendix \ref{sect:Ljets}.  In particular, for the Lie algebroid $\op{Id} \co T_X \to T_X$, we recover (smooth sections of) the infinite jet bundle $J \to X$.

\subsection*{Acknowledgements}
This work benefitted from many discussions with Kevin Costello and Owen Gwilliam and I thank them for their input and suggestions. The work of Damien Calaque and collaborators has similarly provided substantial guidance. I would also like to thank Jim Stasheff for feedback and interest at various stages of this project.

The author gratefully acknowledges support from the Hausdorff Research Institute for Mathematics, especially the trimester program ``Homotopy theory, manifolds, and field theories'' at which some of the research for this paper was performed. He also would like to thank the Perimeter Institute for Theoretical Physics for its hospitality and working conditions during the writing of parts of the current paper. Research at Perimeter Institute is supported by the Government of Canada through Industry Canada and by the Province of Ontario through the Ministry of Economic Development and Innovation.

\section{Preliminaries}

We recall our conventions for $\L8$ algebras and Lie algebroids as the former are a central object of study in the present work and the latter will appear in several examples. The expert can safely skip this section.

\subsection{Curved $\L8$ algebras}

\begin{definition}
Let $A$ be a commutative dg algebra with a nilpotent dg ideal $I$. A {\em curved $\L8$ algebra over $A$} consists of
\begin{enumerate}
\item[(1)] a locally free, $\ZZ$-graded $A^\sharp$-module $V$ and
\item[(2)] a linear map of cohomological degree 1
\[
d: \Sym (V[1]) \to  \Sym (V[1]),
\]
\end{enumerate}
where $\Sym (V[1])$ indicates the graded vector space given by the symmetric algebra over the graded algebra $A^\sharp$ underlying the dg algebra $A$. Further, we require 
\begin{enumerate}
\item[(i)] $d^2 = 0$,
\item[(ii)] $(\Sym (V[1]),d)$ is a cocommutative dg coalgebra over $A$ (i.e., $d$ is a coderivation), and
\item[(iii)] modulo $I$, the coderivation $d$ vanishes on the constants (i.e., on $\Sym^0$).
\end{enumerate}
\end{definition}

To reduce notation, we use $C_\ast(V)$ to denote the cocommutative dg coalgebra $(\Sym (V[1]),d)$;  we call it the \emph{Chevally-Eilenberg homology complex} of $V$, as it extends the usual notion of Lie algebra homology. (Note, in contrast to some conventions, for us, the homology complex of any ordinary Lie algebra $\fg$ is concentrated in {\it nonpositive} degrees.)

There is also a natural Chevalley-Eilenberg \emph{cohomology} complex $C^\ast(V)$. It is $(\csym (V^\vee[-1]),d)$, where the notation $\csym (V^\vee[-1])$ indicates the completed symmetric algebra over the graded algebra $A^\sharp$ underlying the dg algebra $A$.  The differential $d$ is the ``dual" differential to that on $C_*(V)$. In particular, it makes $C^*(V)$ into a commutative dg algebra, so $d$ is a derivation.

\begin{definition}
A {\em map of curved $\L8$ algebras} $\phi: V \to W$ is a map of cocommutative dg coalgebras $\phi_*: C_*(V) \to C_*(W)$ respecting the cofiltration by $I$. A map $\phi$ is a {\em weak equivalence} if its linear component $\phi_1 \co V \to W$ is a quasi-isomorphism of cochain complexes.
\end{definition}

\subsection{Lie Algebroids}

Standard references for Lie algebroids include Mackenzie \cite{Mackenzie} and Rinehart \cite{Rinehart}.

\begin{definition}
A {\it Lie algebroid} on a smooth manifold $X$ is a vector bundle $L \to X$ equipped with the structure of a Lie algebra on its sheaf of smooth sections and an {\it anchor map} $\rho \co L \to T_X$, which is a map of vector bundles, such that
\begin{enumerate}
\item the map on sections induced by $\rho$ is a map of Lie algebras and
\item for $X,Y \in \Gamma (L)$ and $f \in C^\infty_X$, we have the Leibniz rule
\[
[X, fY] = f [X,Y] + (\rho(X) f) Y.
\]
\end{enumerate}
\end{definition}

To any Lie algebroid $\rho \co L \to T_X$ there is an associated commutative dg algebra. We will call it the {\it Chevalley-Eilenberg cohomology complex of $L$} and denote it $\calgd(L)$, because it is modeled on the Chevalley-Eilenberg cochain complex of a Lie algebra. It is often also called the {\it de Rham complex of $L$} because for the Lie algebroid $\id: L = T \to T$, we have $\calgd(L) = \Omega^*_X$, the usual de Rham complex. 

The complex is constructed as follows. Let $L^\vee$ be the dual vector bundle to $L$ and consider the map $d_L \co \Gamma (X,\Lambda^m L^\vee) \to \Gamma (X,\Lambda^{m+1} L^\vee)$ given by
\begin{align*}
\MoveEqLeft[8] (d_L \alpha)(x_0 , \dotsc, x_m) = \frac{1}{m+1} \sum_{k=0}^{m+1} (-1)^k \rho (x_k) \alpha (x_0, \dotsc, \widehat{x_k} , \dotsc , x_m )\\[1ex]
& + \frac{1}{m+1} \sum_{k < l} (-1)^{k+l+1} \alpha ([x_k , x_l], x_0 , \dotsc , \widehat{x_k} , \dotsc , \widehat{x_l} , \dotsc , x_m ).
\end{align*}
Define $\calgd(L)$ to be the cochain complex 
\[
\cinf(X) = \Gamma (X, \Lambda^0 L^\vee) \xto{d_L} \Gamma (X, L^\vee)\xto{d_L} \cdots \xto{d_L} \Gamma (X,\Lambda^{n-1} L^\vee) \xto{d_L} \Gamma (X,\Lambda^n L^\vee)
\]
where $n = \dim X$. We will let $H^\ast_L(X)$ denote the cohomology of $\calgd(L)$.  Notice that the anchor map $\rho \co L \to~T_X$ induces a cochain map $\rho^\vee \co \Omega^\ast_X \to \calgd(L)$ and consequently an algebra map $H^\ast(\rho^\vee) \co H^\ast_{dR} (X)~\to~H^\ast_L(X)$. Just as we may view $\Omega^*_X$ as a sheaf of commutative dg algebras on $X$, we will also use $\calgd(L)$ to denote the sheaf of commutative dg algebras on $X$ given by taking smooth sections on varying open subsets of $X$.

The usual Chern-Weil theory extends to the setting of $L$-connections and the resulting characteristic classes are called primary classes of the Lie algebroid $L$.  These primary classes are always cohomologous to the image under the map $\rho^\ast \co \Omega^\ast_X \to \Omega_L$ of classes that arise by equipping $L$ with a standard connection and applying the standard Chern-Weil construction.  However, there are secondary characteristic classes which are richer invariants of the Lie algebroid structure; see Fernandes \cite{Fernandes} for further discussion.

\section{Parametrized $\L8$ Spaces}\label{sect:param}

In this section we will define {\it paramatrized $\L8$ spaces}. These spaces describe families of formal moduli problems of various flavors. The type of family will be dictated by the choice of a paramatrizing {\it dg nil-thickened manifold}.

\begin{definition} An {\it dg nil-thickened manifold}  is a triple $(M, \cA_M , \cI_M)$ where
\begin{enumerate}
\item $M$ is a smooth manifold;
\item $\cA_M$ is a sheaf of unital commutative differential graded algebras over $\Omega^\ast_M$ such that there exists a vector bundle $A \to M$ so that $\cA_M$ is the sheaf of sections of $A$;
\item There is a map $q: \cA_M \to C^\infty_M$ of sheaves of $\Omega^\ast_M$ algebras whose kernel is $\cI_M$. Further, $\cI_M$ satisfies the following:
\begin{itemize}
\item[(a)] There exists a vector bundle $I$ such that $\cI_M$ is the sheaf of sections of $I$.
\item[(b)] There exists a non-negative integer $n$ such that there is a chain of vector bundles
\[
0=I^{n+1} \subset I^n \subset I^{n-1} \subset \dotsb \subset I,
\]
so that the induced filtration at the level of sections is compatible with the algebra structure, i.e.,
\[
\cI_M^k \cdot \cI_M^l \subset \cI_M^{k+l}.
\]
\end{itemize}
\end{enumerate}
\end{definition}

Note that the existence of the vector bundle $A$ is actually redudant since the kernel of the map $q: \cA_M \to C^\infty_M$ is required to be a vector bundle.  Further, the conditions in (3) mean that $\cI_M$ is a nilpotent dg ideal.

\begin{example}
There are many geometric examples of dg nil-thickened manifold.
\begin{enumerate}
\item For any smooth manifold, the triple $(M, \cinf_M, 0)$ defines a dg nil-thickened manifold.
\item For any smooth manifold, the triple $\left (M, \Omega^\ast_M, \Omega^{\ge 1}_M \right )$ defines a dg nil-thickened manifold.
\item Let $Y$ be a complex manifold, then we can view the Dolbeault complex $\Omega^{0,\ast}_Y$ as the quotient of $\Omega^\ast_Y$ by the ideal generated by $\Omega^{1,\ast}_Y$.  Hence, the triple $\left (Y, \Omega^{0,\ast}_Y , \Omega^{0, \ge 1}_Y \right )$ is a dg nil-thickened manifold.
\item Let $M$ be a smooth manifold and $\rho : L \to T_M$ a Lie algebroid.  The triple $\left (M, C^\ast L , C^{\ge 1} L \right )$ defines a dg nil-thickened manifold.  In particular, to any foliation $\cF \subset T_M$ there is an associated dg nil-thickened manifold. \end{enumerate}
\end{example}

We can now define $\L8$ spaces parametrized by a dg nil-thickened manifold.

\begin{definition}
Let $(X, \cA_X, \cI_X)$ be a dg nil-thickened manifold. 
\begin{enumerate}
\item A {\em curved $\L8$ algebra over $\cA_X$} consists of a $\ZZ$-graded topological vector bundle $\pi: V \to X$ and the structure of a curved $\L8$ algebra structure on its sheaf of smooth sections, denoted $\fg$, where the base algebra is  $\cA_X$ with nilpotent ideal $\cI_X$.
\item An {\em $\L8$ space parametrized by} $(X, \cA_X, \cI_X)$ is a pair $(X, \fg^{\cA})$, where $\fg^\cA$ is a curved $\L8$ algebra over $\cA_X$.
\end{enumerate}
\end{definition}

\section{The associated derived stack}\label{sect:stack}

In \cite{GGLoop} we showed that $\L8$ spaces parametrized by $\left (X, \Omega^\ast_X , \Omega^{\ge 1}_X \right )$ present derived stacks, the same is true of $\L8$ spaces parametrized by a dg nil-thickened manifold $(X, \cA_X, \cI_X)$.

There is a site $\dgMan$ of {\it nil dg manifolds}, in which an object $\cM$ is a smooth manifold $M$ equipped with a sheaf $\sO_\cM$ of commutative dg algebras over $\Omega^*_M$ that has a nil dg ideal $\sI_\cM$ such that $\sO_\cM/\sI_\cM \cong \cinf_M$. For the full definition, including the definition of cover, see \cite{GGLoop}. 

\begin{remark}
In \cite{GGLoop}, we required the objects of $\dgMan$ to satisfy a local Poincar\'{e} lemma, i.e., for sufficiently small $U \subset M$, the cohomology of $\sO_\cM (U)$ is concentrated in non-positive degrees.  This requirement was motivated by classical deformation theory/derived geometry, but is actually unneccessary and can be removed without bother.  
\end{remark}

\begin{definition}\label{def:derivedstack}
A {\it derived stack} is a functor $\bX: \dgMan^{op} \to s\!\Sets$ satisfying
\begin{enumerate}
\item[(1)] $\bX$ sends weak equivalences of nil dg manifolds to weak equivalences of simplicial sets;
\item[(2)] $\bX$ satisfies \v{C}ech descent, i.e., if for every nil dg manifold $\cM$ and every cover $\fV$ of $\cM$, we have a weak equivalence of simplicial sets
\[
\cF(\cM) \xrightarrow{\simeq} \holim_{\check{C}\fV}\,\cF,
\]
where $\check{C}\fV_\bullet$ denotes the \v{C}ech nerve of the cover (namely the simplicial diagram with $n$-simplices $\check{C}\fV_n := \fV \times_\cM \cdots \times_\cM \fV$).
\end{enumerate} 
\end{definition}

We now explain how every $\L8$ space parametrized by $\left (X, \Omega^\ast_X , \Omega^{\ge 1}_X \right )$ defines a derived stack.  Let $(X,\fg)$ be such an $\L8$ space. Observe that given a smooth map $f: Y \to X$, we obtain a curved $\L8$ algebra over $\Omega^*_Y$ by
\[
f^* \fg := f^{-1}\fg \otimes_{f^{-1} \Omega^*_X} \Omega^*_Y,
\]
where $f^{-1} \fg$ denotes the sheaf of smooth sections of the pullback vector bundle $f^{-1} V$.

\begin{definition}
For $(X,\fg)$ an $\L8$ space parametrized by $\left (X,  \Omega^\ast_X, \Omega^{\ge 1}_X \right )$, its {\it functor of points} $\bB\fg: \dgMan^{op} \to s\!\Sets$
sends the nil dg manifold $\cM$ to the simplicial set $\bB\fg(\cM)$ in which an $n$-simplex is a pair $(f,\alpha)$: a smooth map $f: M \to X$ and a solution $\alpha$ to the Maurer-Cartan equation in sections over $M$ of the $\L8$ algebra $f^* \fg \ot_{\Omega^*_M} \sI_\cM \ot_\RR \Omega^*(\triangle^n)$.
\end{definition}

The notation is such as to hint at the connection to classifying spaces.  See \cite{GGLoop} or \cite{Getzler} for further discussion of the functor $\bB \fg$. A central result of \cite{GGLoop} is then the following.

\begin{theorem}[Theorem 4.8 \cite{GGLoop}]\label{thm:L8IsDerived}
The functor $\bB \fg$ associated to an $\L8$ space $(X,\fg)$ is a derived stack.
\end{theorem}

This theorem extends to more general $\L8$ spaces, e.g., where we change the parametrizing dg nil-thickened manifold, as we now describe.

 Let $B\fg^\cA = (X,\fg^\cA)$ denote an $\L8$ space parametrized by $(X, \cA_X, \cI_X)$. Given a smooth map $f: Y \to X$, we obtain a curved $\L8$ algebra over $\Omega^*Y$ by
\[
f^* \fg^\cA := f^{-1}\fg^\cA \otimes_{f^{-1} \Omega^*_X} \Omega^*_Y,
\]
where $f^{-1} \fg^\cA$ denotes sheaf of smooth sections of the pullback vector bundle $f^{-1} V$ and where we equip $f^{-1}\fg^\cA$ as a module over $f^{-1} \Omega^\ast_X$ via restriction of scalars.

\begin{definition}
For $B\fg^\cA = (X,\fg^\cA)$ an $\L8$ space parametrized by $(X, \cA_X, \cI_X)$, its {\em functor of points} is the functor
\[
\bB\fg^\cA: \dgMan^{op} \to s\!\Sets
\]
for which an $n$-simplex of $\bB\fg^\cA(\cM)$ is a pair $(f,\alpha)$: a smooth map $f: M \to X$ and a solution $\alpha$ to the Maurer-Cartan equation in $f^* \fg^\cA \ot_{\Omega^*_M} \sI_\cM \ot_\RR \Omega^*(\triangle^n)$.
\end{definition}

\begin{theorem}\label{thm:L8Aisderived}
The functor $\bB\fg^\cA$ associated to an $\L8$ space $B\fg^\cA$ parametrized by a dg nil-thickened manifold $(X, \cA_X, \cI_X)$ defines a derived stack.
\end{theorem}

\begin{proof}
The proof is the same as Theorem 4.8 of \cite{GGLoop}, restated as Theorem\label{thm:L8IsDerived} above, see Appendix C of {\it op. cit}. There are two key properties which continue to hold given a nilpotent dg manifold $\cM =(M, \sO_\cM, \sI_\cM)$ and a fixed map $f: M \to X$. The first is that  $\bB\fg^\cA (M)$ is the discrete simplicial sheaf of smooth maps $M \to X$. The second observation is that the sheaf $f^\ast \fg^\cA \otimes_{\Omega^\ast_M} \sI_\cM^{k-1}/\sI_\cM^k$ is given as sections of a vector bundle on $M$.  With these two properties, the same Artinian induction argument works as in \it{op. cit.}
\end{proof}

\section{Examples}\label{sect:example}

\subsection{Basic examples over $X_{dR}$}

Let $X$ be a smooth manifold and consider the dg nil-thickened manifold $\left (X, \Omega^\ast_X, \Omega^{\ge 1}_X \right )$. We describe several $\L8$ spaces over this parametrizing space.

\begin{enumerate}
\item The 0 vector space has a canonical $\L8$ algebra structure, so $(X, 0)$ defines an $\L8$ space. Notice that $C^\ast (0) \cong \Omega^\ast_X$ as we are doing linear algebra over the cdga $\Omega^\ast_X$.
\item Consider the $\Omega^\sharp_X$ module
\[
\fg_X = \Omega^\sharp_X (T_X[-1]).
\]
Building on ideas of Kapranov \cite{Kap} and Costello \cite{CosWG2}, we showed in \cite{GG1} (see also \cite{GGLoop} for further discussion) that $(X, \fg_X)$ is an $\L8$ space encoding the smooth structure of $X$. In particular, we have that
\[
\cinf_X \hookrightarrow dR(J) \cong C^\ast (\fg_X)
\]
is a quasi-isomorphism of $\Omega_X$ algebras.
\item Similar to previous example (though preceding it historically), if further $X$ is a complex manifold, then we can complexify forms and use the holormphic tangent bundle to define an $\L8$ algebra
\[
\fg_{X_{\overline{\partial}}} = \Omega^\sharp_X \left (T^{1,0}_X [-1]  \right ).
\]
The $\L8$ space $(X, \fg_{X_{\overline{\partial}}})$ is a {\it derived enhancement} of $X$ as a complex manifold (e.g., $C^\ast (\fg_{X_{\overline{\partial}}}) \simeq \sO_X$) and the associated derived stack (as in the preceding section) represents the moduli of holomorphic maps to $X$.
\item Let $\rho \co L \to T_X$ be a Lie algebroid. We can define an $\L8$ space $(X, \fg_L)$ which is a generalization of examples (1)-(3) above.  Indeed, define
\[
\fg_L = \Omega^\sharp_X (T_X[-1] \oplus L).
\]
We prove in \cite{GGAlgd} that $\fg_L$ is a curved $\L8$ algebra over $\Omega^\ast_X$ such that we have a quasi-isomorphism of $\Omega^\ast_X$ algebras
\[
C^\ast (L) \xrightarrow{\sim} C^\ast (\fg_L).
\]
Note that in {\it op. cit.}, the $\L8$ space $(X, \fg_L)$ is denoted $\op{enh} \, L$.

\end{enumerate}

\subsection{Holomorphic families of formal moduli problems}

Now let $X$ be a complex manifold; we will describe several $\L8$ spaces over $\left (X, \Omega^{0, \ast}_X , \Omega^{0, \ge 1}_X \right)$.  As mentioned previously, one can think of these $\L8$ spaces as describing a family of formal moduli problems which depend holomorphically on $X$.

\begin{enumerate}
\setcounter{enumi}{4}
\item Let $(X, \sO_X) \hookrightarrow (Y, \sO_Y)$ be a closed embedding of complex manifolds; let $N$ denote the normal bundle of this embedding.  In \cite{Yu15}, Yu proves that 
\[
\fg_N \overset{def}{=} \Omega^{0,\sharp} (N[-1])
\]
has the structure of a curved $\L8$ algebra over $\left (X, \Omega^{0,\ast}_X , \Omega^{0, \ge 1}_X \right )$. 
\item Consider the diagonal embedding (for $X$ a complex manifold) $\Delta : X \hookrightarrow X \times X$. In this case the construction of Yu, as in the preceding example, recovers the $\L8$ algebra structure on the shifted holomorphic tangent bundle $\Omega^{0,\sharp} (T[-1])$ originally due to Kapranov \cite{Kap}. 
Alternatively, the same $\L8$ space can be realized by starting from the $\L8$ space $(X, \fg_{X_{\overline{\partial}}})$ of example (3) and using the anti-holomorphic projection onto the $(0,\ast)$ component to obtain an $\L8$ space parametrized by $\left (X, \Omega^{0, \ast}_X , \Omega^{0, \ge 1}_X \right )$. 
\end{enumerate}

\subsection{Descending to $X_{dR}$}

As illustrated by example (6) of the preceding section, an $\L8$ space over the (complex) de Rham space naturally leads to a holomorphic $\L8$ space.  Similarly, for a smooth manifold an $\L8$ space over the de Rham space leads to an uncurved $\L8$ space over $(X, \cinf_X , 0)$.  The converse of the statements is in general obstructed and requires equipping the family of moduli problems with a connection which is flat up to homotopy.

We consider a simple example over $\RR^n$ to simplify explicit comparison, though the same behavior is exhibited over an arbitrary manifold. Further, we will compare the Maurer-Cartan sets--the 0-simplices of the associated Maurer-Cartan spaces--of two explicit dglas to illustrate the difference between working over smooth functions versus the de Rham complex.

Let $\fn$ be a differential graded lie algebra over the algebra of smooth functions on $\RR^n$, e.g., $\fn$ could be given as the sections of a bundle of dglas over $\RR^n$. There is an associated dgla over the de Rham complex 
\[
\widetilde{\fn} := \fn \otimes_{\cinf_{\RR^n}} \Omega^\ast_{\RR^n}, 
\]
with 1-ary bracket (differential) given by $\widetilde{\ell}_1 = \ell_1 \otimes \op{Id} + \op{Id} \otimes d_{dR} $
and the 2-ary bracket $\ell_2$ extended linearly over $\Omega^\ast_{\RR^n}$. There is a surjective map of $\L8$ algebras $q \co \widetilde{\fn} \to \fn$ and we have an induced surjection at the level of Maurer-Cartan sets $\op{MC}(q) \co \op{MC} (\widetilde{\fn}) \to \op{MC} (\fn)$ (actually, following \cite{Getzler} the map $q$ induces a fibration of {\it Maurer-Cartan spaces}). The map $\op{MC}(q)$ does not admit a canonical section, indeed if we view an element $\alpha \in \op{MC}(\fn)$ as a smooth family of solutions to the Maurer-Cartan equation, then a choice of lift $\widetilde{\alpha} \in \op{MC} (\widetilde{\fn})$ is  further data such that $\widetilde{\alpha}$ is flat up to homotopy.

Even more explicitly, the degree 1 part of $\widetilde{\fn}$ is given by
\[
\widetilde{\fn}_1 = \dotsb \oplus \fn_{-1} \{ dx_i \wedge dx_j \} \oplus \fn_0 \{ dx_i\} \oplus \fn_1 .
\]
For $A =A_{1-n} + \dotsb + A_{-1} + A_0 + A_1 \in \widetilde{\fn}_1$ to be a Maurer-Cartan element it needs to satisfy a system of equations (sorted by form degree):
\[
\begin{array}{ll}
\ell_1 (A_1) + \frac{1}{2} \ell_2 (A_1,A_1) &=0\\[1ex]
d_{dR} A_1 + \ell_1 (A_0) + \ell_2 (A_1 , A_0) &=0\\[1ex]
d_{dR} A_0 + \ell_1 (A_{-1}) + \ell_2 (A_1 , A_{-1} ) + \frac{1}{2} \ell_2 (A_0 , A_0) &=0\\[1ex]
\hfil \vdots \hfil & \hfil \vdots \hfil \\
d_{dR} A_{2-n} + \ell_1 (A_{1-n}) + \ell_2 (A_1 , A_{1-n} ) + \ell_2 (A_0 , A_{2-n}) +\dotsb &=0.
\end{array}
\]
The first equation is just the Maurer-Cartan equation for $\fn$, we can make sense of the others by define a differential $\partial = \ell_1 + \ell_2 (A_1, -)$, then the second equation becomes
\[
-d_{dR} A_1 = \partial A_0.
\]
That is, $A_1$ is flat up to an exact correction term.  The second equation then says that the correction $A_0$ satisfies the Maurer-Cartan equation up to a higher correction
\[
-d_{dR} A_0 - \frac{1}{2} \ell_2 (A_0 , A_0) = \partial A_{-1},
\]
and so on and so on.

\section{An extended example: resolving a smooth manifold over a Lie algebroid}\label{sect:algd}

Let $\rho \co L \to T_X$ be a Lie algebroid. We construct an $\L8$ space $B \fg_{X_L} = (X , \fg_{X_L})$ over the augmented dg manifold $\left (X, C^\ast L , C^{\ge 1} L \right )$ as follows.

\begin{prop}\label{prop:BgXL}
Let $\rho : L \to T_X$ be a Lie algebroid.  There is a curved $\L8$ algebra $\fg_{X_L}$ over $C^\ast L$, with nilpotent ideal $C^{\ge 1} L$, such that
\begin{enumerate}
\item $\fg_{X_L} \cong C^\sharp L \otimes_{C^\infty_X} L [-1]$ as $C^\sharp L$ modules;
\item $dR_L (J_L) \cong C^\ast (\fg_{X_L})$ as commutative $C^\ast L$ algebras;
\item We have a quasi-isomorphism of $C^\ast L$ algebras
\[
  C^\infty_X \xrightarrow{\simeq} dR_L (J_L).
\]
\end{enumerate}
\end{prop}

\begin{proof}
This proof mimics that of Lemma 4.12 of \cite{GGLoop} in which case $L = T_X$ and $\rho$ is the identity. We recall some relevant facts from Section 4.2.5 of \cite{CVDB}. The $L$-jets are filtered (as in the standard case) and we have a canonical isomorphism
\[
\Gr J_L = \csym_{C^\infty_X} (L^\vee),
\]
so in particular for each $k \ge 0$ we have an isomorphism
\[
F^k J_L / F^{k+1} J_L = \Sym^k (L^\vee).
\]
By picking a $C^\infty_X$-linear splitting of the map $F^1 J_L \to C^\infty_X$, we obtain--via the universal property of symmetric algebras--an isomorphism
\[
\csym_{C^\infty_X} (L^\vee) \xrightarrow{\cong} J_L .
\]
This isomorphism depends on the choice of splitting, however, the space of splittings is contractible and the resulting $\L8$ structure will be unique up to isomorphism.

Further recall from \cite{CVDB}, that $J_L$ has a flat $L$-connection $\nabla$. Indeed, for $\phi \in J_L$, $D \in UL$, and $\chi \in L$ the connection is given by
\[
\nabla_\chi (\phi ) (D) = \chi \cdot \phi (D) - \phi (\chi D).
\]
Hence, via the isomorphism above we obtain a commutative $C^\ast L$ algebra
\[
C^\ast (\fg_{X_L}) \overset{def}{=} C^\ast L \otimes_{C^\infty_X} \csym_{C^\infty_X} (L^\vee),
\]
which equips $C^\sharp L \otimes_{C^\infty_X} L[-1]$ with the structure of an $\L8$ algebra over $C^\ast L$.

Property (2) is obvious from the construction and (3) is exactly Lemma 4.2.4 of \cite{CVDB}.

\end{proof}

\subsection{Characteristic Classes for $B \fg_{X_L}$}


In this section we define natural characteristic classes for our $\L8$ space $B \fg_{X_L}$. We use a Chern-Weil style approach with Atiyah classes to recover the primary classes of the Lie algebroid $L$.

Let $E$ be a vector bundle over the smooth manifold $X$. Our goal is to find the obstruction to $E$ defining a representation of $L$.  Via Proposition \ref{prop:connfromjets} below, we can rephrase this problem in terms of Atiyah classes.  Let us outline our construction (which follows a similar vein as in Section 11 of \cite{GG1}):
\begin{enumerate}
\item In the previous section, we have resolved $C^\infty_X$ as a $C^\ast L$ algebra using $L$-jets and the $L$-de Rham complex
\[
dR_L (J_L) \simeq C^\infty_X.
\]
\item Similarly, after splitting the canonical quotient map $J_L (\sE) \to \sE$, we can resolve $\sE$ as a $C^\ast L$ algebra
\[ dR_L (J_L (\sE)) \simeq \sE , \]
where the module structure on the right hand side is given by the map $C^\ast L \to \cinf_X$.
\item The dg module $dR_L (J_L (\sE))$ over the dg algebra $dR_L (J_L)$ admits a flat connection.
\item The Atiyah class of this connection measures the failure of the connection to be compatible with the $C^\ast L$ algebra structures.
\item The Chern classes built from the Atiyah class recover the Chern classes of $E$.
\end{enumerate}

To justify step (2) in the outline, we have the following $L$-jet analog of Lemma E.2 of \cite{GGLoop}.

\begin{lemma}\label{lem:jetiso}
Let $\sE$ be a locally free $C^\infty_X$ module and $\nu \co \sE \to J_L (\sE)$ a splitting of the quotient map $p \co J_L (\sE) \to \sE$, then we have a quasi-isomorphism of $dR_L$ algebras
\[
dR_L (J_L (\sE)) \simeq \sE.
\]
Further, we have an isomorphism induced by $\nu$
\[
i_\nu \co J_L \otimes_{\cinf_X} \sE \cong J_L (\sE) .
\]
\end{lemma}

Recall from Proposition \ref{prop:Ljets} that $J_L(\sE)$ is a module over $J_L$. Therefore, $dR_L (J_L (\sE))$ is a dg module over $dR_L (J_L)$.  It turns out that this module has a natural connection.

\begin{prop}\label{prop:dgLconn}
As a differential graded module over $dR_L(J_L)$, $dR_L(J_L(\sE))$ is equipped with a flat connection, i.e., there exists a map
\[
\widetilde{\nabla}_\nu \co dR_L(J_L(\sE)) \to dR_L(J_L(\sE)) \otimes_{dR_L(J_L)} \Omega^1_{dR_L(J_L)}
\]
such that
\[
\widetilde{\nabla}_\nu (e \cdot j) = e \otimes d_{dR} (j) + (-1)^{\lvert j \rvert} \widetilde{\nabla}_\nu (e) \cdot j
\]
and $\widetilde{\nabla}_\nu^2 =0$.
\end{prop}

\begin{proof}
The isomorphism $i_\nu$ of Lemma \ref{lem:jetiso} induces an isomorphism
\[
dR_L (J_L (\sE)) = dR_L \otimes_{\cinf_X} J_L (\sE) \cong_{i_\nu} dR_L \otimes_{\cinf_X} J_L \otimes_{\cinf_X} \sE = dR_L (J_L) \otimes_{\cinf_X} \sE \cong \sE \otimes_{\cinf_X} dR_L (J_L).
\]
Consequently, we have an isomorphism
\[
dR_L(J_L (\sE)) \otimes_{dR_L (J_L)} \Omega^1_{dR_L (J_L)} \cong \sE \otimes_{\cinf_X} \Omega^1_{dR_L (J_L)}.
\]
It is then easy to verify that the map
\[
\widetilde{\nabla}_\nu \co \sE \otimes_{\cinf_X} dR_L (J_L) \xrightarrow{1 \otimes d_{dR}} \sE \otimes_{\cinf_X} \Omega^1_{dR_L (J_L)}
\]
defines a flat connection.
\end{proof}

Now, there is no reason for $\widetilde{\nabla}_\nu$ to be compatible with the internal differential of $dR_L (J_L (\sE))$. Equivalently, $\widetilde{\nabla}_\nu$ may fail to be a map of $dR_L$-modules. Such failure is measured by the Atiyah class
\[
\At (\widetilde{\nabla}_\nu ) \in \Omega^1_{dR_L(J_L)} \otimes_{dR_L (J_L)} \End_{dR_L (J_L)} (dR_L (J_L (\sE))) .
\]
Since our connection $\widetilde{\nabla}_\nu$ is flat, the associated Chern classes are closed (Proposition \ref{prop:atiyahproperties}), i.e., 
\[
ch_k (\widetilde{\nabla}_\nu) \in \Omega^k_{dR_L (J_L),cl} .
\]

\begin{prop}\label{prop:dRL}
We have a quasi-isomorphism of $C^\ast L$-algebras
\[
\Omega^k_{dR_L ( J_L),cl} \simeq C^{\ge k} L.
\]
\end{prop}

\begin{proof}
By Lemma 4.3.6 of \cite{CVDB}, we have an isomorphism of dg-algebras
\[
dR_L (J_L (C^\ast L)) \xrightarrow{\cong} \Omega^\ast_{dR_L (J_L)}.
\]
Now by Lemma \ref{lem:jetiso} (which follows from Lemma 4.2.4 of \cite{CVDB}), we have a quasi-isomorphism $C^\ast L$-algebras
\[
C^\ast L \xrightarrow{\simeq} dR_L (J_L (C^\ast L)).
\]
The lemma now follows from the definition of closed forms, as the relevant double complex degenerates into the truncated complex of cochains $C^{\ge k} L$.
\end{proof}

\begin{prop}\label{prop:connfromjets}
The connection $\widetilde{\nabla}_\nu$ induces an $L$-connection $\nabla_0$ on $E$. Further, the curvature of $\nabla_0$ corresponds to the constant term of the Atiyah class $\At (\widetilde{\nabla}_{\nu})$.
\end{prop}

\begin{proof} 
That $\widetilde{\nabla}_\nu$ induces an $L$-connection $\nabla_0$ is completely analogous to Proposition 11.5 of \cite{GG1}: $\nabla_0$ is given by $p \circ \widetilde{\nabla}_\nu \circ J_L$. The curvature of $\nabla_0$ is a local computation as in Lemma 11.6 of \cite{GG1}.
\end{proof}

\begin{cor}\label{cor:primaryinvariants}
For each $k$, the Chern class $ch_k (\widetilde{\nabla}_{\nu})$ computes the $kth$ Chern class of the bundle $E$, i.e., under the identification of Proposition \ref{prop:dRL} we have
\[
ch_k (\widetilde{\nabla}_{\nu}) = \rho^\ast ch_k (E) \in H^{2k}_L ,
\]
where $\rho \co L \to T_X$ is the anchor map.
\end{cor}

\begin{proof}
At the level of cohomology, the output of the Chern-Weil construction is independent of the choice of connection.  Hence the corollary follows immediately.  Note that when computing primary invariants we can assume without loss of generality that any $L$-connection is actually the pullback of a $T_X$-connection, i.e., the usual notion of connection on the vector bundle $E$ over $X$.
\end{proof}

\appendix

\section{The $L$-jets functor}\label{sect:Ljets}

The $\infty$-jet of a smooth function provides a coordinate-free way to work with its Taylor expansion (more accurately, its Taylor expansion at every point). The jet bundle $J \to X$ is an infinite-rank vector bundle equipped with a canonical flat connection such that a section is flat if and only if it is the $\infty$-jet of a smooth function. Following \cite{CRVDB}, we recall the corresponding construction of the jets of a Lie algebroid $\rho \co L \to T_X$. 

We define the {\it sheaf of $L$-jets} $J_L$ by 
\[
J_L \overset{def}{=} {\mathcal Hom}_{C^\infty_X} ( \cU_L , C^\infty_X),
\]
where ${\mathcal Hom}_{\cinf_X}$ denotes the sheaf-hom of $\cinf_X$-module sheaves and $\cU_L$ is the universal enveloping algebra of the Lie algebroid $L$. Since $\cU_L$ is canonically filtered, so are the $L$-jets $J_L$.  The coalgebra structure on $\cU_L$ equips $J_L$ with the structure of a commutative algebra.  

This definition extends naturally: for $\sE$ a $\cinf_X$-module, we define
\[
J_L (\sE) \overset{def}{=} {\mathcal Hom}_{C^\infty_X} (\cU_L , \sE),
\]
and call it the {\it $L$-jets of $\sE$}.

The following result from Section 5.3 of \cite{CRVDB} will play a useful role for us.

\begin{prop}\label{prop:Ljets}
\mbox{}
\begin{itemize}
\item[(a)] $J_L$ is a commutative $\cinf_X$-module.
\item[(b)] For a $\cinf_X$-module $\sE$, its $L$-jets $J_L (\sE)$ is a $J_L$-module.
\end{itemize}
\end{prop}

The $L$-jets have a natural flat $L$-connection $\nabla$ given by
\[
\nabla_X (\phi)(D) = \rho (X) (\phi (D)) - \phi (X D),
\]
for $X \in L$, $\phi \in J_L$, and $D \in \cU_L$. We then have a commutative $\Omega_L$ dg algebra given by
\[
dR_L (J_L) \overset{def}{=} (\Omega_L \otimes_{C^\infty_X} J_L, \nabla).
\]
More generally, we obtain a functor from left $\cU_L$ modules to left $\Omega_L$ modules which we denote by $dR_L$. It is standard, that we have identifications
\[
\cU_{T_X} \cong D_X \quad \text{and} \quad C^\ast T_X \cong \Omega_X.
\]
So for $L=T_X \xrightarrow{\op{Id}} T_X$, we recover the standard infinity jet and de Rham functors which are basic constructions in the theory of D-modules.

\section{Atiyah Classes}

There is a reformulation of Atiyah classes due to Kapranov \cite{Kap} and Calaque-van den Bergh \cite{CVDB}. It generalizes the original setting of Atiyah \cite{At}. Indeed, if we let $R= (\Omega^{0,\ast}(X), \overline{\partial})$, $M= (\Omega^{0,\ast} (E), \overline{\partial})$, and $\nabla$ be a connection on the holomorphic bundle $E \to X$ which satisfies Leibniz with respect to $\partial$, then the formalism below exhibits the Atiyah class $\At(\nabla)$ as the obstruction to extending $\nabla$ to a holomorphic connection.

Let $R = (R^\#, d)$ be a commutative dg algebra over a base ring $k$. The underlying graded algebra is denoted $R^\#$. We denote the K\"ahler differentials of $R$ by $\Omega^1_R$ and let $d_{dR}: R \rightarrow \Omega^1_R$ denote the universal derivation, which extends to a differential $d_{dR} : \Omega^k_R \to \Omega^{k+1}_R$.

\begin{definition}
Let $M$ be an $R$-module that is projective over $R^\#$. A {\em connection} on $M$ is a $k$-linear map $\nabla \co M \rightarrow M \ot_R \Omega^1_R$ such that
\[
\nabla(r \cdot m) = (d_{dR} r) m + (-1)^{|r|} r \nabla m, 
\]
for all $r \in R$ and $m \in M$.
\end{definition}

A connection may not be compatible  with the differential $d_M$ on $M$, and the Atiyah class is precisely the obstruction to compatibility between $\nabla$ and the dg $R$-module structure on $M$. 

\begin{definition}
The {\em Atiyah class} of $\nabla$ is the class in $\Omega^1_R \ot_R \End_R(M)$ given by
\[
\At(\nabla) = [\nabla, d] = \nabla \circ d_M - d_{\Omega^1_R \ot_R M} \circ \nabla.
\]
\end{definition}

Assume that we have a trace map $\Tr: \End_R (M) \to R$, then we define the {\it Chern character} of $\nabla$ by
\[
ch(\nabla) := \Tr \exp \left( \frac{\At(\nabla)}{-2\pi i} \right).
\]
We let $ch_k(\nabla)$ denote the homogeneous component of $ch(\nabla)$ in $\Omega^k_R$. Hence, 
\[
ch_k(\nabla) = \Tr \left(  \frac{1}{k!(-2\pi i)^k} \At(\nabla)^k \right).
\]
Note that $ch_k (\nabla)$ is an element of degree $k$ in $\Omega^k_R$.

We recall the following from \cite[Corollary 6.7, Proposition 6.8, and Corollary 7.2]{GG1}.

\begin{prop}\label{prop:atiyahproperties} For $\At (\nabla) \in \Omega^1_R \ot_R \End_R(M)$ and $ch_k (\nabla) \in \Omega^k_R$, we have the following.
\begin{itemize}
\item  $\At (\nabla)$ is a cocycle, i.e., $d_{\Omega^1_R \ot_R \End_R (M)} \At (\nabla) =0$;
\item If $\nabla^2 =0$ and $\nabla^{\End}$ is the induced connection on $\End_R (M)$, then $\At (\nabla)$ is horizontal, i.e., $\nabla^{\End} \At (\nabla) =0$;
\item If $\nabla^2 =0$, then $ch_k (\nabla)$ is closed under both $d_{dR}$ and $d_{\Omega_R}$.
\end{itemize}
\end{prop}

\bibliographystyle{amsalpha}
\bibliography{algebroids}

\newcommand{\etalchar}[1]{$^{#1}$}
\providecommand{\bysame}{\leavevmode\hbox to3em{\hrulefill}\thinspace}
\providecommand{\MR}{\relax\ifhmode\unskip\space\fi MR }
\providecommand{\MRhref}[2]{%
  \href{http://www.ams.org/mathscinet-getitem?mr=#1}{#2}
}
\providecommand{\href}[2]{#2}
\begin{thebibliography}{CRvdB10}

\bibitem[Ati57]{At}
M.~F. Atiyah, \emph{Complex analytic connections in fibre bundles}, Trans.
  Amer. Math. Soc. \textbf{85} (1957), 181--207. \MR{0086359 (19,172c)}

\bibitem[Cos]{CosWG2}
Kevin Costello, \emph{A geometric construction of the {W}itten genus, {II}},
  available at \href{http://front.math.ucdavis.edu/1112.0816}{arXiv:1112.0816}.

\bibitem[Cos11]{Cos1}
\bysame, \emph{Renormalization and effective field theory}, Mathematical
  Surveys and Monographs, vol. 170, American Mathematical Society, Providence,
  RI, 2011. \MR{2778558}

\bibitem[CPT{\etalchar{+}}]{CPTVV}
Damien Calaque, Tony Pantev, Bertrand To\"{e}n, Michel Vaqui\'{e}, and Gabriele
  Vezzosi, \emph{Shifted {P}oisson structures and deformation quantization},
  available at the \href{http://arxiv.org/abs/1506.03699}{arXiv:1509.03699}.

\bibitem[CRvdB10]{CRVDB}
Damien Calaque, Carlo~A. Rossi, and Michel van~den Bergh, \emph{Hochschild
  (co)homology for {L}ie algebroids}, Int. Math. Res. Not. IMRN (2010), no.~21,
  4098--4136. \MR{2738352 (2011m:14026)}

\bibitem[CVdB10]{CVDB}
Damien Calaque and Michel Van~den Bergh, \emph{Hochschild cohomology and
  {A}tiyah classes}, Adv. Math. \textbf{224} (2010), no.~5, 1839--1889.
  \MR{2646112 (2011i:14037)}

\bibitem[Fer02]{Fernandes}
Rui~Loja Fernandes, \emph{Lie algebroids, holonomy and characteristic classes},
  Adv. Math. \textbf{170} (2002), no.~1, 119--179. \MR{1929305 (2004b:58023)}

\bibitem[Get09]{Getzler}
Ezra Getzler, \emph{Lie theory for nilpotent {$L\sb \infty$}-algebras}, Ann. of
  Math. (2) \textbf{170} (2009), no.~1, 271--301. \MR{2521116 (2010g:17026)}

\bibitem[GG]{GGAlgd}
Ryan Grady and Owen Gwilliam, \emph{Lie algebroids as ${L}_\infty$ spaces},
  available at.

\bibitem[GG14]{GG1}
\bysame, \emph{One-dimensional {C}hern--{S}imons theory and the \^{A} genus},
  Algebr. Geom. Topol. \textbf{14} (2014), no.~4, 419--497. \MR{3331615}

\bibitem[GG15]{GGLoop}
\bysame, \emph{{$L\sb \infty$} spaces and derived loop spaces}, New York J.
  Math. \textbf{21} (2015), 231--272. \MR{3358542}

\bibitem[Kap99]{Kap}
M.~Kapranov, \emph{Rozansky-{W}itten invariants via {A}tiyah classes},
  Compositio Math. \textbf{115} (1999), no.~1, 71--113. \MR{1671737
  (2000h:57056)}

\bibitem[LM95]{LadaMarkl}
Tom Lada and Martin Markl, \emph{Strongly homotopy {L}ie algebras}, Comm.
  Algebra \textbf{23} (1995), no.~6, 2147--2161. \MR{1327129}

\bibitem[LS93]{LadaJim}
Tom Lada and Jim Stasheff, \emph{Introduction to {SH} {L}ie algebras for
  physicists}, Internat. J. Theoret. Phys. \textbf{32} (1993), no.~7,
  1087--1103. \MR{1235010}

\bibitem[Lur]{LurieDAGX}
Jacob Lurie, \emph{Derived {A}lgebraic {G}eometry {X}: {F}ormal {M}oduli
  {P}roblems}, available at the
  \href{http://www.math.harvard.edu/~lurie/}{author's homepage}.

\bibitem[Lur10]{LurieICM}
\bysame, \emph{Moduli problems for ring spectra}, Proceedings of the
  {I}nternational {C}ongress of {M}athematicians. {V}olume {II}, Hindustan Book
  Agency, New Delhi, 2010, pp.~1099--1125. \MR{2827833}

\bibitem[Mac05]{Mackenzie}
Kirill C.~H. Mackenzie, \emph{General theory of {L}ie groupoids and {L}ie
  algebroids}, London Mathematical Society Lecture Note Series, vol. 213,
  Cambridge University Press, Cambridge, 2005. \MR{2157566 (2006k:58035)}

\bibitem[PV]{PanGab}
Tony Pantev and Gabriele Vezzosi, \emph{Symplectic and {P}oisson derived
  geometry and deformation quantization}, available at the
  \href{http://arxiv.org/abs/1603.02753}{arXiv:1603.02753}.

\bibitem[Rin63]{Rinehart}
George~S. Rinehart, \emph{Differential forms on general commutative algebras},
  Trans. Amer. Math. Soc. \textbf{108} (1963), 195--222. \MR{0154906 (27
  \#4850)}

\bibitem[Tua]{Tu1}
Junwu Tu, \emph{Homotopy l-infinity spaces}, available at
  \href{http://arxiv.org/abs/1411.5115}{arXiv:1411.5115}.

\bibitem[Tub]{Tu2}
\bysame, \emph{Homotopy l-infinity spaces and kuranishi manifolds, i:
  categorical structures}, available at
  \href{http://arxiv.org/abs/1602.00150}{arXiv:1602.00150}.

\bibitem[Yu]{Yu15}
Shilin Yu, \emph{Dolbeault dga and ${L}_\infty$-algebroid of the formal
  neighborhood}, available at
  \href{http://arxiv.org/abs/1507.07528}{arXiv:1507.07528}.

\end{thebibliography}

\end{document}